\documentclass[12pt,reqno]{amsart}
\usepackage{amssymb}

\newtheorem{theorem}{Theorem}

\newtheorem{definition}{Definition}

\newtheorem{lemma}{Lemma}
\newtheorem{proposition}{Proposition}
\newtheorem{remark}{Remark}
\numberwithin{equation}{section}
 \numberwithin{remark}{section}
\numberwithin{definition}{section} \numberwithin{lemma}{section}
\numberwithin{proposition}{section}
 \numberwithin{example}{section}
 \numberwithin{corollary}{section}
\numberwithin{theorem}{section}

\begin{document}
\title[]{Realizability of the group of rational self-homotopy  equivalences}
\author{Mahmoud Benkhalifa}
\email{makhalifa@uqu.edu.sa}
\address{Department of Mathematics. Faculty of Applied Sciences. Umm Al-Qura University. Mekka. Saudi Arabia }
\maketitle
\begin{abstract}
For a 1-connected  CW-complex $X$, let $\mathcal{E}(X)$ denote the
group of homotopy classes of self-homotopy equivalences of $X$.
The aim of this paper is to prove that, for every
$n\in\Bbb N$,  there exists a 1-connected
rational CW-complex $X_{n}$ such that $\mathcal{E}(X_{n})\cong
\underset{2^{n+1}\mathrm{. times }}{\underbrace{\Bbb Z_{2}\oplus\cdots \Bbb \oplus \Bbb Z_{2}}}$.
\end{abstract}

\section{Introduction}

If $X$ is a  1-connected CW-complex, let $\mathcal{E}(X)$ denote
the set of homotopy classes of self-homotopy equivalences of $X$.
It is well-known that $\mathcal{E}(X)$ is a group with respect to
composition of homotopy classes. As pointed out by D. W.  Khan
\cite{K}, a basic problem about self-equivalences  is the
realizability of $\mathcal{E}(X)$, i.e., when for a given group
$G$ there exists a CW-complex X such that $\mathcal{E}(X)\cong G$.

\noindent In this paper we consider   a particular problem asked
by M. Arkowitz and G Lupton in \cite{MA}:  let $G$ be a finite
group, is there a rational 1-connected CW-complex X such that
$\mathcal{E}(X)\cong G$.

\noindent In this case the group $G$ is said to be
\emph{rationally realizable}.

\noindent Our main result says:

\noindent \textbf{Theorem}.  \emph{ The groups $
\underset{2^{n+1}\mathrm{. times }}{\underbrace{\Bbb Z_{2}\oplus\cdots \Bbb \oplus \Bbb Z_{2}}}$ are rationally realizable  for every
$n\in\Bbb N$. }

We will obtain these result working on the theory elaborated by
Sullivan \cite{H} which asserts  that  the homotopy of 1-connected
rational spaces is equivalent to the homotopy theory of
1-connected minimal cochain commutative algebras over the
rationals (mccas, for short). Recall that there exists a
reasonable concept of homotopy among cochain morphisms between two
mccas, analogous in many respects to the topological notion of
homotopy.

\noindent Because of this equivalence we deduce that
$\mathcal{E}(X)\cong \mathcal{E}(\Lambda V,\partial)$, where
$(\Lambda V,\partial)$ is  the mcca associated with $X$ (called
the minimal Sullivan model of $X$) and where $\mathcal{E}(\Lambda
V,\partial)$ denotes the group of self-homotopy equivalences of
$(\Lambda V,\partial)$.  Therefore we can translate our problem to
the following: let $G$ be a finite group. Is there a mcca
$(\Lambda V,\partial)$ such that $\mathcal{E}(\Lambda
V,\partial)\cong G$?

\noindent Note that, in \cite{MA}, M. Arkowitz and G Lupton have
given  examples showing that   $\Bbb Z_{2}$ and $\Bbb Z_{2}\oplus
\Bbb Z_{2}$ are rationally realizable. Recently and by using a
technique radically different from the one used in \cite{MA},  the
author  \cite{ben2}
  showed   that  $\underset{2^{n+1}\mathrm{. times }}{\underbrace{\Bbb Z_{2}\oplus\cdots \Bbb \oplus \Bbb Z_{2}}}$  are rationally realizable for all $n\leq 10$.

\section{The main result}
\subsection{Notion of homotopy for mccas}

Let $(\Lambda (t,dt),d)$ be the free commutative graded algebra on
the basis $\{t,dt\}$ with $\mid t\mid=0$, $\mid dt\mid$=1, and let
$d$ be the differential sending $t\mapsto dt$. Define augmentations:
$$\varepsilon_{0},\varepsilon_{1}:(\Lambda (t,dt),d)\to \Bbb Q\,\,\,\text{ by }\,\,\,\varepsilon_{0}(t)=0,\varepsilon_{1}(t)=1$$
 \begin{definition}
(\cite{H})
Two cochain morphisms  $\alpha_{0},\alpha_{1}:(\Lambda V
,\partial)\to(\Lambda W,\delta)$ are homotopic if there is a cochain
morphism $\Phi:(\Lambda V ,\partial)\to (\Lambda W,\delta)\otimes
(\Lambda (t,dt),d)$  such that  , $i=0,1$. Here $\Phi$ is called a
homotopy from $\alpha_{0}$ to $\alpha_{1}$.
\end{definition}
Thereafter we will need the following lemma.
\begin{lemma}
\label{l3} Let  $\alpha_{0},\alpha_{1}:(\Lambda V^{\leq n+1}
,\partial)\to(\Lambda W^{\leq n+1},\delta)$ be two cochain
morphisms such that $\alpha_{0}=\alpha_{1}$  on $V^{\leq n}.$
Assume that for every generator  $v\in V^{n+1}$ we have:
$$\alpha_{0}(v)=\alpha_{1}(v)+\partial(y_{v})$$
where $y_{v}\in (\Lambda W^{\leq n+1})^{n}$. Then $\alpha_{0}$ and
$ \alpha_{1}$ are homotopic.
\end{lemma}
\begin{proof}
Define $\Phi:(\Lambda V ,\partial)\to (\Lambda V,\partial)\otimes
(\Lambda (t,dt),d)$ by setting:
\begin{equation}\label{18}
\Phi(v)=\alpha_{1}(v)+\partial(y_{v})t-(-1)^{|\partial(y_{v})|}y_{v}dt\,\,\,\,\,\,\,\,\,\mathrm{
and }\,\,\,\Phi=\alpha_{0}\,\,\,\mathrm{ on }\,\,\,V^{\leq n}
\end{equation}
It is clear that $\Phi$ is a cochain algebra satisfying
$(id.\varepsilon_{0})\circ \Phi=\alpha_{1}$,
$(id.\varepsilon_{1})\circ \Phi=\alpha_{0}$
\end{proof}

\subsection{The linear maps $b^{n}$, $n\geq 3$}
\begin{definition}
\label{d4}
Let $(\Lambda V,\partial)$ be a 1-connected  mcca. For every  $n\geq 3$, we define the linear map  $b^{n}: V^{n}\rightarrow
H^{n+1}(\Lambda V^{\leq n-1})$ by setting:
\begin{equation}
b^{n}(v_{n})=[\partial(v_{n})] \label{106}
\end{equation}
Here $[\partial(v_{n})]$  denotes
 the cohomology class of $\partial(v_{n})\in
(\Lambda V^{\leq n-1})^{n+1}$.
\end{definition}
For every    1-connected  mcca $(\Lambda V,\partial)$, the linear
map $b_{n}$ are natural. Namely if $\alpha:(\Lambda V
,\partial)\to(\Lambda W,\delta)$ is a cochain morphism between two
1-connected  mccas, then  the following diagram commutes for all
$n\geq 2$:

\begin{picture}(300,90)(-50,30)
\put(60,100){$V^{n+1}\hspace{1mm}\vector(1,0){140}\hspace{1mm}W^{n+1}$}
 \put(69,76){\scriptsize $b^{n+1}$} \put(238,76){\scriptsize $b'^{n+1}$}
\put(66,96){$\vector(0,-1){38}$} \put(235,96){$\vector(0,-1){38}$}
\put(155,103){\scriptsize $\widetilde{\alpha}^{n+1}$} \put(135,52){\scriptsize
$H^{n+2}(\alpha_{(n)})$} \put(50,46){$H^{n+2}(\Lambda V^{\leq n})
\hspace{1mm}\vector(1,0){100}\hspace{1mm}H^{n+2}(\Lambda W^{\leq n})
\hspace{1mm}$} \put(400,76){\scriptsize $(1)$}
\end{picture}

\noindent where $\widetilde{\alpha}:V^{*}\to W^{*}$ is the graded
homomorphism induced by $\alpha$ on the indecomposables and where
$\alpha_{(n)}:(\Lambda V^{\leq n} ,\partial)\to (\Lambda W^{\leq
n} ,\delta)$ is the restriction of $\alpha$.
\subsection{The groups $\mathcal{C}^{n+1}$, where $n\geq 2$}

\begin{definition}\label{d3}
    Given a 1-connected mcca $(\Lambda V^{\leq
   n+1},\partial)$. Let $\mathcal{C}^{n+1}$ be the subset  of $Aut(V^{n+1})\times \mathcal{E}(\Lambda V^{\leq
   n},\partial)$ consisting of the couples $(\xi^{n+1},[\alpha_{(n)}])$  making  the following
    diagram commutes:

\begin{picture}(300,90)(-50,30)
\put(60,100){$V^{n+1}\hspace{1mm}\vector(1,0){140}\hspace{1mm}V^{n+1}$}
 \put(68,76){\scriptsize $b^{n+1}$} \put(238,76){\scriptsize $b^{n+1}$}
\put(66,96){$\vector(0,-1){37}$} \put(235,96){$\vector(0,-1){37}$}
\put(155,103){\scriptsize $\xi^{n+1}$} \put(145,52){\scriptsize
$H^{n+2}(\alpha_{(n)})$} \put(53,48){$H^{n+2}(\Lambda V^{\leq n})
\hspace{1mm}\vector(1,0){100}\hspace{1mm}H^{n+2}(\Lambda V^{\leq n})$} \put(400,76){\scriptsize $(2)$}
\end{picture}

\noindent where   $Aut(V^{n+1})$ is the group of automorphisms of the vector space $V^{n+1}$.
\end{definition}
Equipped with the composition laws,  the set $\mathcal{C}^{n+1}$
becomes a subgroup of $Aut(V^{n+1})\times \mathcal{E}(\Lambda
V^{\leq
   n},\partial)$.

\begin{proposition}\label{p3}
   There exists  a surjective homomorphism $\Phi^{n+1}:\mathcal{E}(\Lambda V^{\leq
   n+1},\partial)\to\mathcal{C}^{n+1}$ given by the relation:
$$\Phi^{n+1}([\alpha])=(\widetilde{\alpha}^{n+1},[\alpha_{(n)}])$$
\end{proposition}

\begin{remark}
\label{rr1} It is well-known (\cite{H} proposition 12.8) that if two
cochain morphisms $\alpha,\alpha':(\Lambda V^{\leq
   n+1} ,\partial)\to(\Lambda
V^{\leq
   n+1},\partial)$ are homotopic, then they induce the same graded linear
maps on the indecomposables, i.e.,
$\widetilde{\alpha}=\widetilde{\alpha'}$, moreover
$\alpha_{(n)},\alpha'_{(n)}$ are homotopic and by using the diagram (1) we deduce that   the map $\Phi^{n+1}$
is
well-defined.\\
\end{remark}
\begin{proof}
Let $(\xi^{n+1},[\alpha_{(n)}])\in \mathcal{C}^{n+1}$. Choose
$(v_{\sigma})_{\sigma\in\Sigma}$ as a basis of $V^{n+1}$. Recall
that, in the diagram (2), we have:
\begin{eqnarray}
\label{117}H^{n+2}(\alpha_{(n)})\circ
b^{n+1}(v_{\sigma})&=&\alpha_{(n)}\circ\partial(v_{\sigma})+\mathrm{Im}\,\partial_{\leq n}\nonumber\\
b^{n+1}\circ
\xi^{n+1}(v_{\sigma})&=&\partial\circ\xi^{n+1}(v_{\sigma})+\mathrm{Im}\,\partial_{\leq
n}\label{103}
\end{eqnarray}
where $\partial_{\leq n}:(\Lambda V^{\leq n})^{n+1} \to (\Lambda
V^{\leq n})^{n+2} $. Note that here we have used the relation
(\ref{106}).

\noindent Since by definition \ref{d3} this diagram commutes, the
element
$(\alpha_{(n)}\circ\partial-\partial\circ\xi^{n+1})(v_{\sigma})\in\mathrm{Im}
\,\partial_{\leq n}$. As a consequence there exists $u_{\sigma}\in
(\Lambda V^{\leq n})^{n+1}$ such that:
\begin{eqnarray}
(\alpha_{(n)}\circ\partial-\partial\circ\xi^{n+1})(v_{\sigma})=\partial_{\leq
n}(u_{\sigma})\label{123}.
\end{eqnarray}
Thus we define $\alpha:(\Lambda V^{\leq n+1}
,\partial)\rightarrow(\Lambda V^{\leq n+1},\partial)$ by setting:
\begin{eqnarray}
\alpha(v_{\sigma})=\xi^{n+1}(v_{\sigma})+u_{\sigma}\,\,\,,\,\,
v_{\sigma}\in V^{n+1}\,\,\,\,\,\text{and}\,\,\,\,\,
\alpha=\alpha_{(n)}\text{ on }V^{\leq
n}.\label{113}
\end{eqnarray}
As $\partial(v_{\sigma})\in (\Lambda V^{\leq n})^{n+2}$ then, by
(\ref{123}), we get:
\begin{eqnarray*}
\partial\circ\alpha(v_{\sigma})=\partial(\xi^{n+1}(v_{\sigma}))+\partial_{\leq n}(u_{\sigma})=\alpha_{(n)}\circ\partial(v_{\sigma})=
\alpha\circ\partial(v_{\sigma})
\end{eqnarray*}
So  $\alpha$ is a  cochain morphism.   Now due to the fact that
$u_{\sigma}\in (\Lambda V^{\leq n})^{n+1}$,   the linear map
$\widetilde{\alpha}^{n+1}:V^{n+1}\to V^{n+1}$ coincides with
$\xi^{n+1}$.

\noindent Finally it  is well-known (see \cite{H}) that  any
cochain morphism between two 1-connected mccas inducing a graded
linear isomorphism on the indecomposables is a homotopy
equivalence. Consequently  $\alpha\in \mathcal{E}(\Lambda V^{\leq
   n+1},\partial)$. Therefore  $\Phi^{n+1}$ is surjective.

\noindent Finally the following relations:
$$\Phi^{n+1}([\alpha].[\alpha'])=(\widetilde{\alpha\circ\alpha'}^{n+1},[\alpha_{(n)}\circ \alpha'_{(n)}])=(\widetilde{\alpha}^{n+1},[\alpha_{(n)}])\circ (\widetilde{\alpha'}^{n+1},[\alpha'_{(n)}])=\Phi^{n+1}([\alpha])\circ\Phi^{n+1}([\alpha'])$$
assure that $\Phi^{n+1}$ is a homomorphism of groups
\end{proof}
\begin{remark}\label{r4}
    Assume that
    $(\alpha_{(n)}\circ\partial-\partial\circ\xi^{n+1})(V^{n+1})\cap
    \partial_{\leq n}
\big((\Lambda V^{\leq n})^{n+1}\big)=\{0\}$,
    then the element $u_{\sigma}\in
(\Lambda V^{\leq n})^{n+1}$, given in the formula (\ref{123}),  must be
a cocycle. Therefore if there are no trivial coycles belong to $
(\Lambda V^{\leq n})^{n+1}$, then the cochain  isomorphism  $\alpha$
defined in (\ref{113}) will satisfy
$\alpha(v_{\sigma})=\xi^{n+1}(v_{\sigma})$, so it is unique. Hence,
in this case, the map $\Phi^{n+1}$ is an isomorphism.
\end{remark}

\subsection{Main theorem}

\noindent For every natural $n\in\Bbb N$, let us consider the following 1-connected mcca:

 $\Lambda V=\Lambda
(x_{1},\ldots,x_{n+2},y_{1},y_{2},y_{3},w,z)$ with
$|x_{n+2}|=2^{n+2}-2$, $|x_{k}|=2^{k}$ for every $1\leq k\leq
n+1$. \noindent The differential is as follows:
 $$ \partial(x_{1})=\cdots=\partial(x_{n+2})=0 \,\,\,\,\,,\,\,\,\,\, \partial(y_{1})=x_{n+1}^{3}x_{n+2} \,\,\,\,\,,\,\,\,\,\,\partial(y_{2})=x_{n+1}^{2}x_{n+2}^{2} $$
 $$\partial(y_{3})=x_{n+1}x_{n+2}^{3}\,\,\,\,\,,\,\,\,\,\,\partial(w)=x_{1}^{28}x_{2}^{18}x_{3}^{18}\ldots x_{n}^{18}$$
$$\partial(z)=x_{1}^{2^{n}+7}(y_{1}y_{2}x_{n+2}^{3}-y_{1}y_{3}x_{n+1}x_{n+2}^{2}+
y_{2}y_{3}x_{n+1}^{2}x_{n+2})+\underset{k=1}{\overset{n+1}{\sum}}
x_{k}^{9.2^{n+2-k}}+x_{1}^{9}x_{n+2}^{9}$$ So that:
$$|y_{1}|=5.2^{n+1}-3\,\,\,\,,\,\,\,|y_{2}|=6.2^{n+1}-5\,\,\,\,\,,\,\,\,|y_{3}|=7.2^{n+1}-7\,\,\,\,\,,\,\,\,|w|=9.2^{n+2}-17\,\,\,\,\,,\,\,\,|z|=9.2^{n+2}-1$$
\begin{theorem}
\label{t1} $\mathcal{E}(\Lambda
V,\partial)\cong\underset{2^{n+1}}{\oplus}\Bbb
 Z_{2}$.
\end{theorem}
Thereafter we will need the following facts.
\begin{lemma}
\label{l1} There are no cocycles (except  0) in $(\Lambda V^{\leq
i-1})^{i}$ for $i=5.2^{n+1}-3$, $6.2^{n+1}-5$, $7.2^{n+1}-7$.
\end{lemma}
\begin{proof}
First since  the generators $x_{k},1\leq k\leq n+2$, have even
degrees we deduce that $(\Lambda V^{\leq
5.2^{n+1}-4})^{5.2^{n+1}-3}=0.$

\noindent Next the vector space $(\Lambda V^{\leq
6.2^{n+1}-6})^{6.2^{n+1}-5}$ has only two generators namely
$y_{1}x^{2^{n}-1}_{1},y_{1}x_{1}x_{2}\ldots x_{n}$ and because of:
$$\partial(y_{1}x^{2^{n}-1}_{1})=x_{n+1}^{3}x_{n+2}x_{1}^{2^{n+1}-1}\,\,\,,\,\,\,\partial(y_{1}x_{1}x_{2}\ldots x_{n})=x_{n+1}^{3}x_{n+2}x_{1}x_{2}\ldots x_{n}$$
we deduce that there  are no cocycles (except  0) in   $(\Lambda V^{\leq
6.2^{n+1}-6})^{6.2^{n+1}-5}$.

\noindent Finally
 $(\Lambda V^{\leq 7.2^{n+1}-8})^{7.2^{n+1}-7}$
is spanned by:
$$y_{1}x^{2^{n+1}-2}_{1}\,,\,y_{1}x^{2^{n}-1}_{2}\,,\,y_{1}x_{1}^{2}x_{2}^{2}\ldots x_{n}^{2}\,,\,y_{2}x^{2^{n}-1}_{1}\,,\,y_{2}x_{1}x_{2}\ldots x_{n}$$
and since   we have:
$$\partial(y_{1}x^{2^{n+1}-2}_{1})=x_{n+1}^{3}x_{n+2}x^{2^{n+1}-2}_{1}\,\,,\,\,\partial(y_{1}x^{2^{n}-1}_{2})=
x_{n+1}^{3}x_{n+2}x^{2^{n}-1}_{2}\,\,,\,\,\partial(y_{2}x^{2^{n}-1}_{1})=x_{n+1}^{2}x_{n+2}^{2}x^{2^{n}-1}_{1},$$
$$\partial(y_{1}x_{1}^{2}x_{2}^{2}\ldots x_{n}^{2})=x_{n+1}^{3}x_{n+2}x_{1}^{2}x_{2}^{2}\ldots x_{n}^{2}\,\,\,,\,\,\,\partial(y_{2}x_{1}x_{2}\ldots x_{n})=x_{n+1}^{2}x_{n+2}^{2}x_{1}x_{2}\ldots x_{n} $$
we conclude that  there  are no cocycles (except  0) belonging to  $(\Lambda V^{\leq
7.2^{n+1}-8})^{7.2^{n+1}-7}$.
\end{proof}

\begin{lemma}
\label{l2} Every  cocycles in $(\Lambda V^{\leq
9.2^{n+2}-2})^{9.2^{n+2}-1}$ is a coboundary.
\end{lemma}
\begin{proof}
First an easy computation shows that  $(\Lambda V^{\leq
9.2^{n+2}-2})^{9.2^{n+2}-1}$ is generated by the elements on the form:\\

$
\begin{array}{cc}
 y _{1}x_{1}^{a_{1}}x_{2}^{a_{2}}...x_{n+1}^{a_{n+1}}x_{n+2}^{a_{n+2}} & \text{ where  }\,\,\, \underset{i=1}{\overset{n+2}{\sum}}
a_{i}2^{i}-2a_{n+2}=13.2^{n+1}+2 \\
  y_{2}x_{1}^{b_{1}}x_{2}^{b_{2}}...x_{n+1}^{b_{n+1}}x_{n+2}^{b_{n+2}} & \text{ where  }\,\,\, \underset{i=1}{\overset{n+2}{\sum}}
b_{i}2^{i}-2b_{n+2}=12.2^{n+1}+4 \\
 y_{3}x_{1}^{c_{1}}x_{2}^{c_{2}}...x_{n+1}^{c_{n+1}}x_{n+2}^{c_{n+2}} & \text{ where  }\,\,\, \underset{i=1}{\overset{n+1}{\sum}}
c_{i}2^{i}-2c_{n+2}=11.2^{n+1}+6 \\
x_{1}^{e_{1}}x_{2}^{e_{2}}x_{3}^{e_{3}}y_{1}y_{2}y_{3}  & \hspace{-19mm}\text{ where  }\,\,\,
e_{1}+2e_{2}+4e_{3}=7 \\
  wx_{1}^{d_{1}}x_{2}^{d_{2}}x_{3}^{d_{3}}x_{4}^{d_{4}} & \hspace{-8mm}\text{ where  }\,\,\, d_{1}+2d_{2}+4d_{3}+8d_{4}=8
\end{array}
$
\vspace{5mm}

\noindent Since:
\begin{eqnarray*}
\partial(x_{1}^{e_{1}}x_{2}^{e_{2}}x_{3}^{e_{3}}y_{1}y_{2}y_{3})&=& x_{1}^{e_{1}}x_{2}^{e_{2}}x_{3}^{e_{3}}(x_{n+1}^{3}x_{n+2}y_{2}y_{3}-x_{n+1}^{2}x_{n+2}^{2}y_{1}y_{3}+x_{n+1}x_{n+2}^{3}y_{1}y_{2})\\
\partial(wx_{1}^{d_{1}}x_{2}^{d_{2}}x_{3}^{d_{3}}x_{4}^{d_{4}})&=&wx_{1}^{28+d_{1}}x_{2}^{18+d_{2}}x_{3}^{18+d_{3}}x_{4}^{18+d_{4}}x_{5}^{18}\ldots x_{n}^{18}
\end{eqnarray*}
we deduce that  the elements which could be cocycles  in $(\Lambda V^{\leq 9.2^{n+2}-2})^{9.2^{n+2}-1}$ are on the form:  $$\alpha
y_{1}x_{1}^{a_{1}}x_{2}^{a_{2}}...x_{n+1}^{a_{n+1}}x_{n+2}^{a_{n+2}}+\beta
y_{2}x_{1}^{b_{1}}x_{2}^{b_{2}}...x_{n+1}^{b_{n+1}}x_{n+2}^{b_{n+2}}
+\lambda
y_{3}x_{1}^{c_{1}}x_{2}^{c_{2}}...x_{n+1}^{c_{n+1}}x_{n+2}^{c_{n+2}}$$
with   the following relations:
$$a_{i}=b_{i}=c_{i}\,\,\,\,\,\,\,\,,\,\,\,\,\,\,\,\,\,\, 1\leq i\leq n\,\,\,\,\,\,\,\,,\,\,\,\,\,\,\,\,\,\,\alpha+\beta+\lambda=0$$
$$c_{n+1}=a_{n+1}+2\,\,\,\,\,,\,\,\,\,\, c_{n+2}=a_{n+2}-2\,\,\,\,\,\,\,\,,\,\,\,\,\,\,\,\,\,\, b_{n+1}=a_{n+1}+1\,\,\,\,\,,\,\,\,\,\, b_{n+2}=a_{n+2}-1$$
Accordingly  the elements:
$$y_{1}x_{1}^{a_{1}}x_{2}^{a_{2}}...x_{n+1}^{a_{n+1}}x_{n+2}^{a_{n+2}}-y_{3}x_{1}^{a_{1}}x_{2}^{a_{2}}...x_{n+1}^{a_{n+1}+2}x_{n+2}^{a_{n+2}-2}$$
$$y_{2}x_{1}^{a_{1}}x_{2}^{a_{2}}...x_{n+1}^{a_{n+1}+1}x_{n+2}^{a_{n+2}-1}-y_{3}x_{1}^{a_{1}}x_{2}^{a_{2}}...x_{n+1}^{a_{n+1}+2}x_{n+2}^{a_{n+2}-2}$$
with  $\underset{i=1}{\overset{n+1}{\sum}}
a_{i}2^{i}+a_{n+2}(2^{n+2}-2)=13.2^{n+1}+2$, span the space of cocycles in $(\Lambda V^{\leq
9.2^{n+2}-2})^{9.2^{n+2}-1}$.

\noindent Finally due to:
\begin{eqnarray*}
\label{21}
  \partial(y_{1}y_{3}x_{1}^{a_{1}}x_{2}^{a_{2}}...x_{n+1}^{a_{n+1}-1}x_{n+2}^{a_{n+2}-3}) &=& -y_{1}x_{1}^{a_{1}}x_{2}^{a_{2}}...x_{n+1}^{a_{n+1}}x_{n+2}^{a_{n+2}}+y_{3}x_{1}^{a_{1}}x_{2}^{a_{2}}...x_{n+1}^{a_{n+1}+2}x_{n+2}^{a_{n+2}-2} \\
  \partial(y_{2}y_{3}x_{1}^{a_{1}}x_{2}^{a_{2}}...x_{n+1}^{a_{n+1}}x_{n+2}^{a_{n+2}-4}) &=&- y_{2}x_{1}^{a_{1}}x_{2}^{a_{2}}...x_{n+1}^{a_{n+1}+1}x_{n+2}^{a_{n+2}-1}+y_{3}x_{1}^{a_{1}}x_{2}^{a_{2}}...x_{n+1}^{a_{n+1}+2}x_{n+2}^{a_{n+2}-2}\nonumber
\end{eqnarray*}
we deduce that $(\Lambda V^{\leq 9.2^{n+2}-2})^{9.2^{n+2}-1}$  is
generated by  coboundaries and the lemma is proved.
\end{proof}
By the same manner  we have:
\begin{lemma}
\label{l22} The sub-vector space of cocycles in $(\Lambda V^{\leq
9.2^{n+2}-18})^{9.2^{n+2}-17}$ is generated by the  elements on the
form:
$$
y_{1}x_{1}^{a'_{1}}x_{2}^{a'_{2}}...x_{n+1}^{a'_{n+1}}x_{n+2}^{a'_{n+2}}-y_{3}x_{1}^{a'_{1}}x_{2}^{a'_{2}}...x_{n+1}^{a'_{n+1}+2}x_{n+2}^{a'_{n+2}-2}$$
$$y_{2}x_{1}^{a'_{1}}x_{2}^{a'_{2}}...x_{n+1}^{a'_{n+1}+1}x_{n+2}^{a'_{n+2}-1}-y_{3}x_{1}^{a'_{1}}
x_{2}^{a'_{2}}...x_{n+1}^{a'_{n+1}+2}x_{n+2}^{a'_{n+2}-2}$$
where $\underset{i=1}{\overset{n+2}{\sum}}
a'_{i}2^{i}-2a'_{n+2}=13.2^{n+1}-14$. Moreover each generator
of $(\Lambda V^{\leq 9.2^{n+2}-18})^{9.2^{n+2}-17}$  is a coboundary.
\end{lemma}

\begin{remark}
\label{r5} We have  the following  elementary facts:

\noindent  1)     Any isomorphism $\xi^{i}:V^{i}\to V^{i}$, where
$i=2,\ldots,2^{n+1}$, $2^{n+2}-2$, $5.2^{n+1}-3$, $6.2^{n+1}-5$, $7.2^{n+1}-7$,
$ 9.2^{n+2}-17$ and $ 9.2^{n+2}-1$, is a multiplication with a
nonzero rational number, so we write:
$$\xi^{2}=p_{1}\,\,\,\,\,,\,\,\,\,\,\xi^{4}=p_{2},\,\,\,\,\,\,\,\,\,\,\ldots\ldots\,\,\,\,\,\,\,\,\,\,,\,\,\,\,\,\,\,\,\,\,\xi^{2^{n+2}}=p_{n+2}$$
$$\xi^{5.2^{n+1}-3}=p_{y_{1}}\,\,\,\,\,,\,\,\,\,\,\xi^{6.2^{n+1}-5}=p_{y_{2}}\,\,\,\,\,,\,\,\,\,\,\xi^{7.2^{n+1}-7}=p_{y_{3}}\,\,\,\,\,,\,\,\,\,\,\xi^{9.2^{n+2}-17}=p_{w}\,\,\,\,\,,\,\,\,\,\,\xi^{9.2^{n+2}-1}=p_{z}$$\\

\noindent  2)   As the generators:
$$ x_{n+1}^{3}x_{n+2}\,\, ,\,\,  x_{n+1}^{2}x_{n+2}^{2}
  \,\, ,\,\,x_{n+1}x_{n+2}^{3}\,\, ,\,\,x_{1}^{28}x_{2}^{18}x_{3}^{18}\ldots x_{n}^{18}\,\, ,\,\,x_{1}^{9.2^{n+1}}\,\, ,\,\,x_{2}^{9.2^{n}},\ldots
,\,\, ,\,\,x_{n+1}^{9}\,\, ,\,\,x_{n+1}^{9.2}\,\, ,\,\,x_{1}^{9}x_{n+2}^{9}$$
$$x_{1}^{2^{n}+7}(y_{1}y_{2}x_{n+2}^{3}-y_{1}y_{3}x_{n+1}x_{n+2}^{2}+y_{2}y_{3}x_{n+1}^{2}x_{n+2})$$
are not reached by the differential and by the definition of the linear maps $b^{n}$  we deduce that:

$$
\begin{array}{cccccc}
 b^{5.2^{n+1}-3}(y_{1}) & = & x_{n+1}x_{n+2}^{3} & b^{9.2^{n+2}-17}(w) & = & x_{1}^{28}x_{2}^{18}x_{3}^{18}\ldots x_{n}^{18} \\
  b^{7.2^{n+1}-7}(y_{3}) & = & x_{n+1}^{3}x_{n+2}& b^{6.2^{n+1}-5}(y_{2}) & = & x_{n+1}^{2}x_{n+2}^{2}
\end{array}
$$
$$b^{9.2^{n+2}-1}(z)  = x_{1}^{2^{n}+7}(y_{1}y_{2}x_{n+2}^{3}-y_{1}y_{3}x_{n+1}x_{n+2}^{2}+
y_{2}y_{3}x_{n+1}^{2}x_{n+2})+\underset{k=1}{\overset{n+1}{\sum}}
x_{k}^{9.2^{n+2-k}}+x_{1}^{9}x_{n+2}^{9}$$

\noindent  3)   Since the differential is nil on the generators
$x_{k}$, for every $1\leq k\leq n+2$, any cochain isomorphism
$\alpha_{(k)}:(\Lambda V^{\leq k},\partial)\to (\Lambda V^{\leq
k},\partial)$  can be written as follows:
\begin{equation}\label{11}
\alpha_{(k)}(x_{k})=p_{k}x_{k}+\sum
q_{_{m_{1},m_{2},\cdots,m_{k-1}}}x_{1}^{m_{1}}x_{2}^{m_{2}}...x_{k-1}^{m_{k-1}}
\end{equation}
where:
$$\underset{i=1}{\overset{k-1}{\sum}}
m_{i}2^{i}=2^{k}\,\,\,\,\, \,\,\,\,\,,\,\,\,\,\,
p_{k},q_{_{m_{1},m_{2},\cdots,m_{k-1}}}\in \Bbb Q\,\,\,\,\,
\,\,\,\,\,,\,\,\,\,\, p_{k}\neq 0$$
\end{remark} Now the last pages
are devoted to the proof of theorem \ref{t1}.

\begin{proof} Let us begin by computing the group
\textbf{$\mathcal{E}(\Lambda V^{\leq 5.2^{n+1}-3},\partial)$}.
Indeed, by  remark \ref{r4} and lemma \ref{l1} the homomorphism
$\Phi^{ 5.2^{n+1}-3}:\mathcal{E}(\Lambda V^{\leq
   5.2^{n+1}-3},\partial)\to\mathcal{C}^{5.2^{n+1}-3}$, given by proposition \ref{p3},  is an isomorphism. So,  by definition \ref{d3}, we have
   to determine all   the couples
   $(\xi^{5.2^{n+1}-3},[\alpha_{(5.2^{n+1}-4)}])\in Aut(V^{5.2^{n+1}-3})\times \mathcal{E}(\Lambda V^{\leq
   5.2^{n+1}-4},\partial)$  such that:
\begin{equation}
\label{5} b^{5.2^{n+1}-3}\circ
\xi^{5.2^{n+1}-3}=H^{5.2^{n+1}-2}(\alpha_{(5.2^{n+1}-4)})\circ
b^{5.2^{n+1}-3}
\end{equation}
 Indeed; since $V^{\leq
   5.2^{n+1}-3}=V^{\leq
   2^{n+2}-2}$ we deduce that, on the generators $x_{k},1\leq k\leq n+2$, the cochain morphism $\alpha_{(5.2^{n+1}-4)}$ is given by the
   relations (\ref{11}).  Therefore:
\begin{eqnarray}
H^{5.2^{n+1}-2}(\alpha_{(5.2^{n+1}-4)})\circ
b^{5.2^{n+1}-3}(y_{1})&=&\Big(p_{n+1}x_{n+1}+\sum
q_{_{m_{1},m_{2},\cdots,m_{n}}}x_{1}^{m_{1}}x_{2}^{m_{2}}...x_{n}^{m_{n}}\Big)^{3}\Big(p_{n+2}x_{n+2}+\nonumber\\
&&\sum
q_{_{m'_{1},m'_{2},\cdots,m'_{n+1}}}x_{1}^{m'_{1}}x_{2}^{m'_{2}}...x_{n+1}^{m'_{n+1}}\Big)\nonumber\\
b^{5.2^{n+1}-3}\circ
\xi^{5.2^{n+1}-3}(y_{1})&=&p_{y_{1}}x_{n+1}^{3}x_{n+2}
\end{eqnarray}
hence  we deduce that
$p_{y_{1}}=p^{3}_{n+1}p_{n+2}$ and that  all the numbers
$q_{_{m_{1},m_{2},\cdots,m_{n}}}$ and
$q_{_{m'_{1},m'_{2},\cdots,m'_{n+1}}}$, given in (\ref{11}),
 should be  nil. Thus we can say that the group
$\mathcal{E}(\Lambda V^{\leq 5.2^{n+1}-3},\partial)$ is consisting
of the classes $[\alpha_{(5.2^{n+1}-3)}]$ such that the cochain
isomorphisms $\alpha_{(5.2^{n+1}-3)}$ satisfy:
$$\alpha_{(5.2^{n+1}-4)}(x_{n+1})=p_{n+1}x_{n+1}\,\,\,,\,\,\alpha_{(5.2^{n+1}-4)}(x_{n+2})=p_{n+2}x_{n+2}\,\,\,,\,\,\alpha_{(5.2^{n+1}-4)}(y_{1}) =p_{y_{1}}y_{1}$$
\begin{eqnarray}
 \label{102}
\alpha_{(5.2^{n+1}-3)}(x_{k}) =p_{k}x_{k}+\sum
q_{_{m_{1},m_{2},\cdots,m_{k-1}}}x_{1}^{m_{1}}x_{2}^{m_{2}}...x_{k-1}^{m_{k-1}}\,\,\,,\,\,1\leq k\leq n
 \end{eqnarray}
with $p_{y_{1}}=p^{3}_{n+1}p_{n+2}$.\\

\noindent \underline{Computation of the group
$\mathcal{E}(\Lambda V^{\leq 6.2^{n+1}-5},\partial)$}\\

This group can be computed from $\mathcal{E}(\Lambda V^{\leq
5.2^{n+1}-3},\partial)$ by using  proposition \ref{p3}. Indeed; by remark \ref{r4} the
homomorphism $\Phi^{6.2^{n+1}-5}:\mathcal{E}(\Lambda V^{\leq
   6.2^{n+1}-5},\partial)\to\mathcal{C}^{6.2^{n+1}-5}$  is also an isomorphism. Recalling again that  the group  $\mathcal{C}^{6.2^{n+1}-5}$
   contains  all  the couples $(\xi^{6.2^{n+1}-5},[\alpha_{(6.2^{n+1}-6)}])$
   such that:
\begin{equation}\label{13}
H^{6.2^{n+1}-4}(\alpha_{(6.2^{n+1}-6)})\circ
b^{6.2^{n+1}-5}=b^{6.2^{n+1}-5}\circ \xi^{6.2^{n+1}-5}.
\end{equation}
Since $\alpha_{(6.2^{n+1}-6)}=\alpha_{(5.2^{n+1}-3)}$ on $V^{\leq
6.2^{n+1}-6}=V^{\leq 5.2^{n+1}-3}$, then  by using (\ref{102})  and the formula giving $b^{6.2^{n+1}-5}$  in remark \ref{r5} we get:
\begin{eqnarray}
H^{6.2^{n+1}-4}(\alpha_{(6.2^{n+1}-6)})\circ
b^{6.2^{n+1}-4}(y_{2})&=&p_{n+1}^{2}p_{n+2}^{2}x_{n+1}^{2}x_{n+2}^{2}\nonumber\\
b^{5.2^{n+1}-3}\circ
\xi^{5.2^{n+1}-3}(y_{2})&=&p_{y_{2}}x_{n+1}^{2}x_{n+2}^{2}
\end{eqnarray}
From the relation (\ref{13}) we deduce that
$p_{y_{2}}=p_{n+1}^{2}p_{n+2}^{2}$. Thus  the group
$\mathcal{E}(\Lambda V^{\leq 6.2^{n+1}-5},\partial)$ is consisting
of all the classes $[\alpha_{(6.2^{n+1}-5)}]$ such that the
cochain isomorphisms $\alpha_{(6.2^{n+1}-5)}$ satisfy:
\begin{eqnarray}
 \label{1021}
 \alpha_{(6.2^{n+1}-5)}(y_{2}) =p_{y_{2}}y_{2} \,\,\,\,\,\,\,\,,\,\,\,\,\,\,\,\,\,\,\,\,\, \alpha_{(6.2^{n+1}-5)}=\alpha_{(5.2^{n+1}-3)}\,\,,\text{ on } V^{\leq 5.2^{n+1}-3}
 \end{eqnarray}
with $p_{y_{2}}=p^{2}_{n+1}p^{2}_{n+2}$.

\noindent \underline{Computation of the group
$\mathcal{E}(\Lambda V^{\leq 7.2^{n+1}-7},\partial)$}\\

\noindent First the same arguments show that $\mathcal{E}(\Lambda
V^{\leq 7.2^{n+1}-7},\partial)$ is isomorphic to  the group
$\mathcal{C}^{7.2^{n+1}-7}$  of  all the couples
$(\xi^{7.2^{n+1}-7},[\alpha_{(7.2^{n+1}-8)}])$
   such that:
\begin{equation}\label{15}
H^{7.2^{n+1}-6}(\alpha_{(7.2^{n+1}-8)})\circ
b^{7.2^{n+1}-7}=b^{7.2^{n+1}-7}\circ \xi^{7.2^{n+1}-7}
\end{equation}
Next since $\alpha_{(7.2^{n+1}-8)}=\alpha_{(6.2^{n+1}-5)}$ on
$V^{\leq 7.2^{n+1}-8}=V^{\leq 6.2^{n+1}-5}$,  we get:
\begin{eqnarray}
H^{7.2^{n+1}-6}(\alpha_{(7.2^{n+1}-8)})\circ
b^{7.2^{n+1}-7}(y_{3})&=&p_{n+1}^{}p_{n+2}^{3}x_{n+1}x_{n+2}^{3}\nonumber\\
b^{7.2^{n+1}-7}\circ \xi^{7.2^{n+1}-7}&=&p_{y_{3}}x_{n+1}x_{n+2}^{3}
\end{eqnarray}
and from  (\ref{15}) we get the
equation $p_{y_{3}}=p_{n+1}p_{n+2}^{3}$. This implies that
$\mathcal{E}(\Lambda V^{\leq 7.2^{n+1}-7},\partial)$
 is consisting of all the classes
$[\alpha_{(7.2^{n+1}-7)}]$ such that the cochain isomorphisms
$\alpha_{(7.2^{n+1}-7)}$ satisfy:
\begin{eqnarray}
 \label{1022}
 \alpha_{(7.2^{n+1}-7)}(y_{3})=p_{y_{3}}y_{3},\,\,\,\,\,\,\,\,,\,\,\,\,\,\,\,\alpha_{(7.2^{n+1}-7)} = \alpha_{(6.2^{n+1}-5)}
 \,\,\,\,\,\,,\,\text{ on }\,\,\,\,V^{\leq 6.2^{n+1}-5}
 \end{eqnarray}
with  $p_{y_{3}}=p_{n+1}p^{3}_{n+2}$.\\

\noindent \underline{The group
$\mathcal{E}(\Lambda V^{\leq 9.2^{n+2}-17},\partial)$}\\
Let us  determine  the group $\mathcal{C}^{9.2^{n+2}-17}$ of
all the  couples $(\xi^{9.2^{n+2}-17},[\alpha_{(9.2^{n+2}-16)}])$  such that:
\begin{equation}\label{333}
H^{9.2^{n+2}-16}(\alpha_{(9.2^{n+2}-18)})\circ
b^{9.2^{n+2}-17}=b^{9.2^{n+2}-17}\circ \xi^{9.2^{n+2}-17}
\end{equation}
Note that $\alpha_{(9.2^{n+2}-18)}=\alpha_{(7.2^{n+1}-7)}$ on
$V^{\leq 9.2^{n+1}-18}=V^{\leq 7.2^{n+1}-7}$. So we deduce that:
\begin{eqnarray}
H^{9.2^{n+2}-16}(\alpha_{(9.2^{n+2}-18)})\circ
b^{9.2^{n+2}-17}(w)&=&p_{1}^{28}x_{1}^{28}.\underset{k=2}{\overset{n}\prod}\Big(p_{k}x_{k}+\sum
q_{_{m_{1},m_{2},\cdots,m_{k-1}}}x_{1}^{m_{1}}x_{2}^{m_{2}}...x_{k-1}^{m_{k-1}}\Big)^{18}\nonumber\\
b^{9.2^{n+2}-17}\circ
\xi^{9.2^{n+2}-17}(w)&=&p_{w}x_{1}^{28}x_{2}^{18}x_{3}^{18}\ldots
x_{n}^{18}
\end{eqnarray}
Now from the relation (\ref{333})  we deduce that
$p_{w}=p_{1}^{38}p_{2}^{18}\ldots p_{n}^{18}$ and that  all the
numbers $q_{_{m_{1},m_{2},\cdots,m_{k-1}}}$, given in (\ref{11}),
 should be  nil.

\noindent  Now by proposition \ref{p3} we have:
$$(\Phi^{9.2^{n+2}-17})^{-1}(\mathcal{C}^{9.2^{n+2}-17})=\mathcal{E}(\Lambda
V^{\leq 9.2^{n+2}-17},\partial)$$
so, by going back to the relation (\ref{113}), we can say that if  $[\alpha]\in\mathcal{E}(\Lambda
V^{\leq 9.2^{n+2}-17},\partial)$, then:
\begin{eqnarray}
 \label{2}
 \alpha(w)&=&p_{w}w+a
 \end{eqnarray}
 where  $a\in(\Lambda V^{\leq 9.2^{n+2}-18})^{ 9.2^{n+2}-17}$. A simple computation shows that:
$$(\alpha\circ\partial-\partial\circ\xi^{9.2^{n+2}-17})(V^{9.2^{n+2}-17})\cap \partial_{\leq 9.2^{n+2}-18}
\big((\Lambda V^{\leq 9.2^{n+2}-18})^{9.2^{n+2}-17}\big)=\{0\}.$$
Therefore by remark \ref{r4}   the element $a$ is a cocycle. But
lemma \ref{l22} asserts that any cocycle in $(\Lambda V^{\leq
9.2^{n+2}-18})^{ 9.2^{n+2}-17}$
  is a coboundary.

\noindent    Thus summarizing our above analysis we infer that  the cochain isomorphisms
$\alpha$ satisfy:
\begin{eqnarray}
 \label{1}
 \alpha(w)&=&p_{w}w+\partial(a')\,\,\,\,\,,\,\,\,\,\text{ where }\partial(a')=a\\
 \alpha&=&\alpha_{(7.2^{n+1}-7)}\,\,\,\,\,,\,\,\,\,\,\,\,\,\,\text{ on }V^{\leq
 9.2^{n+1}-18}\nonumber
\end{eqnarray}
Finally by lemma \ref{l3} all these cochain isomorphisms form one
homotopy class which we represent  by  the cochain isomorphism
 denoted  $\alpha_{(9.2^{n+2}-17)}$ and  satisfying:
 \begin{eqnarray}
 \label{20}
 \alpha_{(9.2^{n+2}-17)}(w)&=&p_{w}w\,\,\,\,\,\,,\,\,\,\,\alpha_{(9.2^{n+2}-17)}(x_{k})=p_{k}x_{k}\,\,\,\,,\,\,\,\,1\leq
k\leq n
 \\
 \alpha_{(9.2^{n+2}-17)}(y_{1})&=&p_{y_{1}}y_{1}\,\,\,\,,\,\,\,\,\alpha_{(9.2^{n+2}-17)}(y_{2})=p_{y_{2}}y_{2}\,\,\,\,,\,\,\,\, \alpha_{(9.2^{n+2}-17)}(y_{3})=p_{y_{3}}y_{3}\nonumber
 \end{eqnarray}
 with:
\begin{eqnarray}
 \label{111}
 p_{y_{1}}=p^{3}_{n+1}p_{n+2}\,\,\,\,,\,\,\,\,p_{y_{2}}=p^{2}_{n+1}p^{2}_{n+2}\,\,\,\,,\,\,\,\,
 p_{y_{3}}=p_{n+1}p^{3}_{n+2}\,\,\,\,,\,\,\,\,p_{w}=p_{1}^{28}p_{2}^{18}\ldots
 p_{n}^{18}.
 \end{eqnarray}

\noindent \underline{Computation of the group
$\mathcal{E}(\Lambda V^{\leq 9.2^{n+2}-1},\partial)$}\\

\noindent  $\mathcal{C}^{9.2^{n+2}-1}$ is the group of  all the
couples $(\xi^{9.2^{n+2}-1},[\alpha_{(9.2^{n+2}-2)}])$
   such that:
\begin{equation}\label{155}
H^{9.2^{n+2}}(\alpha_{(9.2^{n+2}-2)})\circ
b^{9.2^{n+2}-1}=b^{9.2^{n+2}-1}\circ \xi^{9.2^{n+2}-1}.
\end{equation}
Due to the fact that
$\alpha_{(9.2^{n+2}-2)}=\alpha_{(9.2^{n+2}-17)}$ on $V^{\leq
9.2^{n+1}-2}=V^{\leq 9.2^{n+1}-17}$, we deduce that
$\alpha_{(9.2^{n+1}-2)}$ satisfies the relations (\ref{20}).
Consequently:
\begin{eqnarray}
\label{101}
  b^{9.2^{n+2}-1}\circ
\xi^{9.2^{n+2}-1}(z) &=&
p_{z}x_{1}^{2^{n}+7}(y_{1}y_{2}x_{n+2}^{3}-y_{1}y_{3}x_{n+1}x_{n+2}^{2}+
y_{2}y_{3}x_{n+1}^{2}x_{n+2}) \nonumber\\
   & +& \underset{k=1}{\overset{n+1}{\sum}}
p_{z}x_{k}^{9.2^{n+2-k}}+p_{z}x_{1}^{9}x_{n+2}^{9} \nonumber\\
  H^{9.2^{n+2}}(\alpha_{(9.2^{n+2}-2)})\circ
b^{9.2^{n+2}-1}(z) &=&
p_{1}^{2^{n}+7}p_{n+1}^{5}p_{n+2}^{6}x_{1}^{2^{n}+7}(y_{1}y_{2}x_{n+2}^{3}-y_{1}y_{3}x_{n+1}x_{n+2}^{2}+
y_{2}y_{3}x_{n+1}^{2}x_{n+2}) \nonumber\\
   & +&\underset{k=1}{\overset{n+1}{\sum}}
p_{k}^{9.2^{n+2-k}}x_{k}^{9.2^{n+2-k}}+p_{1}^{9}p_{n+2}^{9}x_{1}^{9}x_{n+2}^{9}
\end{eqnarray}
Therefore  from the formulas  (\ref{155}) and (\ref{101}) we
deduce the following equations:
$$p_{z}=p_{1}^{2^{n}+7}p_{n+1}^{5}p_{n+2}^{6}=p_{1}^{9.2^{n+1}}=\ldots =p_{n}^{9.2^{2}}=p_{n+1}^{9.2^{}}=p_{1}^{9}p_{n+2}^{9}.$$

\noindent  Again by proposition \ref{p3} we have:
$$(\Phi^{9.2^{n+2}-1})^{-1}(\mathcal{C}^{9.2^{n+2}-1})=\mathcal{E}(\Lambda
V^{\leq 9.2^{n+2}-1},\partial)$$ so, by going back to the relation
(\ref{113}), if  $[\beta]\in\mathcal{E}(\Lambda V^{\leq
9.2^{n+2}-1},\partial)$, then $\beta(z)=p_{z}z+c$ where, by using
remark
  \ref{r4},   the element $c$ is a cocycle in $(\Lambda V^{\leq 9.2^{n+2}-2})^{ 9.2^{n+2}-1}$.  By lemma \ref{l22}  any cocycle  is a coboundary. Thus  the cochain morphism
$\beta$ satisfy:
\begin{eqnarray}
 \label{1}
 \beta(z)&=&p_{z}z+\partial(c')\,\,\,\,\,,\,\,\,\,\text{ where }\partial(c')=c\\
 \beta&=&\alpha_{9.2^{n+2}-17}\,\,\,\,\,\,\,\,\,,\,\,\,\,\,\,\,\text{ on }V^{\leq
 9.2^{n+2}-2}.\nonumber
 \end{eqnarray}
 Due to lemma  \ref{l3} all these cochain isomorphisms form one homotopy class  which we represent by $\alpha_{(9.2^{n+2}-1)}$ and  satisfying:\\
$
 \begin{array}{ccccccccc}
   \alpha_{(9.2^{n+2}-1)}(z) & = &p_{z}z \,,& \alpha_{(9.2^{n+2}-1)} (w)& = & p_{w}w \,,& \alpha_{(9.2^{n+2}-1)}(x_{k}) & =& p_{k}x_{k}\,,\,1\leq k\leq n+2\\
   \alpha_{(9.2^{n+2}-1)}(y_{1}) & = & p_{y_{1}}y_{1}\,, & \alpha_{(9.2^{n+2}-1)}(y_{2}) & = & p_{y_{2}}y_{2} \,,& \alpha_{(9.2^{n+2}-1)}(y_{3}) & = &\hspace{-3cm} p_{y_{3}}y_{3}
 \end{array}
$

\noindent with the following equations:
$$p_{y_{1}}=p_{n+1}^{3}p_{n+2}\,\,,\,\,p_{y_{2}}=p_{n+1}^{2}p_{n+2}^{2}\,\,,\,\,p_{y_{3}}=p_{n+1}p_{n+2}^{3}\,\,,\,\,p_{w}=p_{1}^{28}p_{2}^{18}\ldots
p_{2}^{18}$$
$$p_{z}=p_{1}^{2^{n}+7}p_{n+1}^{5}p_{n+2}^{6}=p_{1}^{9.2^{n+1}}=\ldots =p_{n}^{9.2^{2}}=p_{n+1}^{9.2^{}}=p_{1}^{9}p_{n+2}^{9}$$
which have the following solutions:
$$p_{n+2}=p_{y_{2}}=p_{w}=1\,\,\,\,\,\,\,\,,\,\,\,\,\,\,\,\,p_{z}=p_{y_{1}}=p_{y_{3}}=p_{1}=p_{2}=\ldots =p_{n}=p_{n+1}=\pm 1.$$
So we distinguish two cases:

 \noindent First case: when $p_{n+1}=1$, then:
 $$p_{n+2}=p_{y_{2}}=p_{w}=1\,\,\,\,\,\,\,\,,\,\,\,\,\,\,\,\,p_{z}=p_{y_{1}}=p_{y_{3}}=p_{n+1}=1 \,\,\,\,\,\,\,\,,\,\,\,\,\,\,\,\,p_{1}=p_{2}=\cdots =p_{n}=\pm 1.$$
 so we find $2^{n}$ homotopy classes.

\noindent Second case: when $p_{n+1}=-1$, then:
$$p_{n+2}=p_{y_{2}}=p_{w}=1\,\,\,\,\,\,\,\,,\,\,\,\,\,\,\,\,p_{z}=p_{y_{1}}=p_{y_{3}}=p_{n+1}=-1 \,\,\,\,\,\,\,\,,\,\,\,\,\,\,\,\,p_{1}=p_{2}=\cdots =p_{n}=\pm 1.$$
 and  we also find $2^{n}$ homotopy classes.  Hence, in total,  we get   $2^{n-1}$ homotopy classes  which are of order 2 (excepted the classs of the identity )
in the group $\mathcal{E}(\Lambda V^{\leq 9.2^{n+2}-1},\partial)$.

\noindent In conclusion we conclude that:
$$\mathcal{E}(\Lambda
V^{\leq 9.2^{n+2}-1},\partial) \cong\underset{2^{n+1}\mathrm{.
times }}{\underbrace{\Bbb Z_{2}\oplus\cdots \Bbb \oplus \Bbb
Z_{2}}}$$
 Now by  the fundamental theorems of rational homotopy
theory  due to Sullivan \cite{H} we can find a  1-connected
rational CW-complex $X_{n}$ such that
$\mathcal{E}(X_{n})\cong\mathcal{E}(\Lambda V^{\leq
9.2^{n+2}-1},\partial) \cong\underset{2^{n+1}\mathrm{. times
}}{\underbrace{\Bbb Z_{2}\oplus\cdots \Bbb \oplus \Bbb Z_{2}}}$.
\end{proof}
\begin{remark}
The  spaces $X_{n}$ are infinite-dimensional CW-complexes:
rational homology is non-zero in infinitely many degrees and, as
rational spaces, with infinitely many cells in each degree in
which they have non-zero homology.
\end{remark}
  We close this work by conjecturing that for   a  1-connected rational CW-complex $X$, if  the group  is not trivial, then $\mathcal{E}(X)$ is either infinite or $\mathcal{E}(X)
\cong \underset{2^{n}\mathrm{. times }}{\underbrace{\Bbb
Z_{2}\oplus\cdots \Bbb \oplus \Bbb Z_{2}}}$ for a certain natural
number $n$.

\end{document}